\documentclass[12pt]{amsart}

\usepackage{amsmath,amsthm}

\newtheorem{theorem}{Theorem}[section]

\newtheorem{lemma}[theorem]{Lemma}

\newtheorem{proposition}[theorem]{Proposition}

\def\C{\mathbb C}
\def\R{\mathbb R}

\begin{document}

\title[quasi-invariant processes generated by projections]{Functional description of a class of quasi-invariant determinantal processes}

%
\author{Roman Romanov}
\address{Department of Mathematics and Computer Science, Saint-Petersburg State University, Saint-Petersburg, Russia}
\email{morovom@gmail.com}
%
\begin{abstract}
We give a functional characterization of a class of quasi-invariant determinantal processes corresponding to projection kernels in terms of de Branges spaces of entire functions. 
\end{abstract} 
\maketitle

\section{Introduction}

This paper is aimed at the first part of the following question posed by G. Olshanski in \cite{GO-Adv}. 

\medskip

\textit{Given a group $ G $ of homemorphisms and a determinantal measure $ \mathbf P $ on a space of configurations $ \operatorname{Conf}(X) $, how to test whether $ \mathbf P $ is $ G $-quasi-invariant? Is it possible to decide this by comparing the correlation kernels $ K (x, y) $ and $ K (g^{ -1 } x, g^{ -1 } y) $?}

\medskip

The group $ G $, depending on the context, is going to be the group of compactly supported permutations, or diffeomorphisms. We understand the first part of the question as the problem of characterization of quasi-invariant determinantal processes in terms of their correlation kernels. Of special interest are processes for which the corresponding correlation kernel is a projection operator in some space $ L^2 $ with respect to a measure \cite{Go-Gelf,Olsh}. In this situation apparently the only known wide class of quasi-invariant determinantal processes was introduced in \cite{Buf-gibbs}. This class contains many previously studied examples -- the sine-, Bessel- and Airy-processes \cite{Dyson,TW-Airy,TW-Bessel}, gamma process \cite{BO-gamma, GO-Adv}, discrete sine-process \cite{BOO}, discrete Bessel process \cite{BOO,johansson}. The description of the class in \cite{Buf-gibbs} is given in geometric terms -- a process is shown to be invariant if the range of the correlation kernel admits certain division property (see below). Here 
we give a complete functional description of quasi-invariant determinantal processes corresponding to projection kernels from the class under discussion. The main results are theorems \ref{main-inv} and \ref{main}.  

Let $ \mu $ be a Borel measure on $ \R $ supported on a Borel subset $ U \subset \R $. Assume that $ \mu $ either has no point masses, or has support discrete in $ \R $. We will refer to these cases as continuous/discrete respectively. Let $ H $ be a reproducing kernel Hilbert space of functions on $ U $ isometrically embedded in $ L^2 ( \R , d\mu ) $, having a real reproducing kernel, and non-degenerate in the sense that for any $ z \in U $ there exists an $ f \in H $ such that $ f(z) \ne 0 $. In the continuous case assume additionally that the corresponding reproducing kernel is locally trace (the definitions are collected in subsection 1.1).

The key definition is as follows. A reproducing kernel Hilbert space $ X $ of functions on a set $ V \subset \R $ satisfies the \textit{division property} if for any $ f \in X $ and any $ k \in V $ such that $ f(k)=0 $ there exists a unique function $ g\in X $ such that 
\[ f(t) = (t-k ) g (t) , \; t \in V . \]

\begin{theorem}\label{main-inv}
Let the space $ H $ satisfy the division property. Then the reproducing kernel $ K $ of the space $ H $ can be represented in the form 
\begin{eqnarray} K ( x, y ) = \Phi ( x ) K_{\mathcal H } (x , y ) \Phi ( y ) , \; x, y \in U \label{K} \end{eqnarray}
where $ \Phi $ is a nonzero function on $ U $, and $ K_{\mathcal H } $ is the reproducing kernel of the de Branges space $ \mathcal H ( E ) $ corresponding to an Hermite--Biehler function $ E $ not vanishing on the real axis. If de Branges spaces $ \mathcal H ( E_1 ) $ and $ \mathcal H ( E_2 ) $ correspond to the same kernel $ K $, then $ K_{\mathcal H_1 } (x , y ) = W(y) K_{\mathcal H_2 } (x , y ) W( x ) $, $ x, y \in \mathbb R $, for some real entire function $ W $ without zeroes.
\end{theorem}

The motivation of this theorem comes from the following result. Recall  \cite{ST} that under the assumptions made $ K ( x, y ) $ is a kernel of some determinantal process. 

\begin{theorem}\cite[Theorems 1.4 and 1.6]{Buf-gibbs}\label{Bu-in}
Let $ H $ satisfy the division property, and additionally

(continuous case): $ U $ be open, $ K ( x, y ) $ be a sufficiently smooth function, and 
\[ \int_U \frac{ K ( x, x) }{ 1+x^2 } d\mu ( x) < +\infty . \]

(discrete case): $ \mu $ be the counting measure, $ \sum_{ t \in U } ( 1+t^2 )^{ -1 } $ be finite, and the space $ H $ contain no compactly supported functions. 

Then the determinantal process corresponding to the kernel $ K $ is quasi-invariant.
\end{theorem}

Precise smoothness assumptions on $ K $ in the continuous case are not specified in \cite{Buf-gibbs} but $ C^2 ( U \times U ) $ will suffice, see also \cite{Bu-Dy-Os}. For such kernels theorem \ref{main-inv} becomes a characterization, in the sense that the division property is satisified for spaces $ H $ with the reproducing kernels of the form \eqref{K} with $ \Phi \in C^2 ( U ) $, $ U $ open, subject to the rest of the assumptions in the continuous case of theorem \ref{Bu-in}, and thus such $ K $'s are correlation kernels of quasi-invariant determinantal processes. We are going to formulate the precise assertion elsewhere, for the actual regularity requirements which ensure the conclusion of theorem \ref{Bu-in} are certainly lower than $ C^2 $, while theorems \ref{main-inv} and \ref{main} have no regularity assumptions at all. 

The proof of theorem \ref{main-inv} is based on the characterization of spaces with division property given by the following theorem. Let $ U \subset \mathbb{R}$ be a nonempty Borel set, and $ H $ be a reproducing kernel Hilbert space of functions on $ U $.

\begin{theorem}\label{main}
Let $ H $ be non-degenerate and satisfy the division property. Assume that the space $ H $ is isometrically embedded into the space $ L^2 ( U, d\mu ) $ for some Borel measure $ \mu $ on $ U $, and 
\begin{eqnarray}\label{tra} \int_{I \cap U} K ( t ,t ) d\mu ( t ) < \infty \end{eqnarray} for any bounded interval $ I $.

Then there exists a function $ \Phi\colon U \longrightarrow \mathbb{C} $, $ \Phi ( x ) \ne 0 $ for all $ x\in U $, and an Hermite--Biehler function $ E $ without real zeroes such that for any $ f \in H $ the function $ f/\Phi $ extends uniquely to an entire function, $ X_f $, from the de Branges space $ \mathcal H ( E ) $, and the map $ f \mapsto X_f $ is an isomorphism of $ H $ on $ \mathcal H (E) $. 
\end{theorem}

Some comments on the theorem are to be made.

(i) The division property is trivially satisfied when $ H $ is a de Branges space $ \mathcal H ( E ) $ with the function $ E $ not having real zeroes. It does not depend on the isometric embedding into an $ L^2 ( \mathbb R, d\mu ) $ space and is clearly preserved by the multiplication by a nonzero function with the corresponding change in the measure. The point of theorems \ref{main} and \ref{main-inv} is that, under mild technical assumptions, all the spaces satisfying the division property are coming from de Branges spaces in this way.

(ii) The uniqueness requirement in the division property ensures that for any $ k \in U $ the function $ \delta_k $ equal to 1 at $ t=k $ and $0 $ otherwise does not belong to $ H $.

(iii) The conclusion of the theorem implies that in its assumptions $ U $ is a uniqueness set for the de Branges space $ \mathcal H $.

(iv) The proof is constructive -- we provide a procedure restoring the de Branges space $ \mathcal H ( E ) $ from the correlation kernel. In section 4 we illustrate this by calculating the Hermite--Biehler function $ E $ in several widely studied cases of quasi-invariant determinantal processes.
This is one of the major distinctions of theorem 1.3 from known results. Let us comment on this. The idea that regular simple symmetric operators with deficiency indices $(1,1)$ are modelled by the multiplication operators in de Branges spaces is a part of the folklore in the de Branges spaces/spectral theory community. The precise assertion most relevant in our context appeared in \cite[Theorem 5.2.6]{Martin}. It refers to the case $ U = \mathbb R $, assumes densely defined multiplication, and, instead of the division property, that the multiplication operator is regular simple with the deficiency indices $ (1,1) $ right away. The proof of \cite[Theorem 5.2.6]{Martin} uses the Beurling-Lax description of invariant subspaces of the shift operator which is non-constructive in that it requires knoweldge of the orthogonal complement to the range of a restricted shift operator, which is not easy to calculate.

(v) One of possible applications of theorem \ref{main-inv} is the interpolation between the different quasi-invariant processes.

\subsection{Notations and basics.} For a subset $ M $ of the real line or the complex plane $ \chi_M $ stands for the indicator function of the set $ M $. The notation is going to be abused by using the same symbol $ \chi_M $ for the operator of multiplication by the function $ \chi_M $ in various function spaces. Given a measure $ \mu $ supported on a Borel set $ U \subset \mathbb R $, a kernel $ P ( x, y ) $, $ x , y \in U $, is said to be \textit{locally trace} if the integral operator with the kernel $ \chi_{M\cap U} (x) P (x,y) \chi_{M \cap U } (y) $ in the space $ L^2 ( U , d\mu ) $ belongs to the trace class for any compact subset $ M \subset \R $.

A point process $ \mathbf P $ with points from a set $ U \subset \R $ is called \textit{determinantal} if there exists a Borel measure $ \mu $ on the set $ U $ and a locally trace class operator $ K $ in the space $ L^2 ( U , d\mu ) $ such that for any bounded measurable function $ g $ on $ U $ for which $ g-1 $ vanishes outside the set $ B = U \cap \{ |x| \le R_g \} $ for some $ R_g > 0 $, the following identity holds,
\[ \mathbb E \bigl( \prod_{ x \in X } g ( x ) \bigr) = \det ( I + ( g-1) K \chi_B ) , \]
where $ \mathbb E $ is the expectation with respect to the measure $ \mathbf P $.

A determinantal process is called \textit{quasi-invariant} with respect to the group $ G $ if the measures $ \mathbf P $ and $ \mathbf P \circ g $ are equivalent for any $ g \in G $. 

An entire function $ G $ is called \textit{Hermite-Biehler function} if $ | G ( z ) | > | G ( \overline z ) | $ whenever $ \Im z > 0 $. Let $ E $ be an Hermite--Biehler function. The \textit{de Branges space} $ \mathcal H ( E ) $ \cite{deBr} is the linear set of entire functions $ f $ such that the functions $ f/ E $ and $ f^* / E $, $ f^* ( z ):= \overline{ f ( \overline z )} $, belong to the Hardy class $ H^2_+ $ in the upper half plane, endowed with metric \[ \left\| f \right\|_{ \mathcal H ( E ) }^2 = \int_\R \left| \frac{f(t)}{E(t)} \right|^2 dt . \]

The \textit{deficiency indices} of a (not necessarily densely defined) symmetric operator $ T $ in a Hilbert space are $ n_\pm ( T ) = \dim\operatorname{Ran} ( T \mp i )^\perp $. We refer to \cite{Krasnosel, de Snoo} for basic notions related to the theory of symmetric operators. A symmetric operator with deficiency indices $ (1,1) $ admits a selfadjoint extension. This is standard for densely defined operators, and can be found in \cite{Krasnosel} or \cite[Corollary 1.7.13]{de Snoo} in the general case, with a very similar proof via Cayley transform. 

The reproducing kernel $ K ( x , y ) $ is sometimes denoted $ K_x ( y ) $ to emphasize the point at which the evaluation takes place.

\medskip 

\textbf{Remark.} In our previous paper \cite{BLMS} the division property of this text was called the weak division property. Since no other division property is going to be used here, we drop the weak from the definition.

\subsection{Outline of the argument.}

The proof of theorem \ref{main} proceeds as follows. First we establish that the operator of multiplication by the independent variable in $ H $ is regular with deficiency indices $ (1,1) $. Then we use a model, due to M. Krein \cite{GG}, of symmetric operators with deficiency indices $(1,1)$ in a Hilbert space of meromorphic functions. When applied to the multiplication operator, the passage to the model is realized by the division by an element of the deficiency subspace. After multiplication by the canonical product over the set of poles, this space turns into a reproducing kernel Hilbert space $ \mathcal X $ of entire functions. Now an axiomatic characterization of the de Branges spaces \cite{deBr} is used. The characterization is as follows.

\medskip 

\textit{Let $ \mathfrak H $ be a reproducing kernel Hilbert space of entire functions such that 
\\
\phantom{oo}(a) for any $ f \in \mathfrak H $ the function $ f^* \in \mathfrak H $, and $ \left\| f \right\| = \left\| f^* \right\| $;
\\
\phantom{oo}(b) for any $ f \in \mathfrak H $ and any $ a\in \C $ such that $ f(a) = 0 $ the function $ f /( \cdot - a ) \in \mathfrak H $, and $ \| f (\cdot - \overline a ) /( \cdot - a ) \| = \| f \| $. 
\\
\phantom{oo}Then there exists an Hermite-Biehler function $ E $ having no real zeroes such that $ \mathfrak H = \mathcal H ( E ) $.}

\medskip

The space $ \mathcal X $ is shown to have property (b). In general, it does not have the property (a), but a symmetrization via multiplication by an appropriate zero-free entire function turns it into the one having both properties, and thus a de Branges space. The possibility of such a symmetrization is actually an abstract fact of the de Branges spaces theory \cite[Problem 54]{deBr}. We provide the required argument, first, to make the overall proof of the theorem constructive and closed, and, second, because the problems in \cite{deBr} are given without solutions.

\section{Proof of Theorem \ref{main}}

Let $ \mathcal D = \{ f \in H \colon tf \in H \} $, and let $ A $ be the operator in $ H $ defined on $ \mathcal D $ by $ (Af)(k) = k f(k) $. The assumptions of theorem \ref{main} imply that $ A $ is a closed symmetric operator in $ H $, with closedness following from the reproducing kernel property. The linear set $ \mathcal D $ may or may not be dense in $ H $, however 

\begin{lemma}
$ \operatorname{dim} \mathcal D^\perp \le 1 $.
\end{lemma}

\begin{proof}
Assume by contradiction that $ \operatorname{dim} \mathcal D^\perp \ge 2 $. Then for any $ a \in U $ there exists a nonzero $ f \in H $, $ f \perp \mathcal D $, such that $ f ( a) = 0 $. By the division property $ f = ( \cdot - a ) h $ for some $ h \in \mathcal D $. 
Also, for any $ g \in H $ such that $ g ( a ) = 0 $ we would have $ g / ( \cdot - a ) \in \mathcal D $, hence
\[ 0 = \left\langle f , \frac g{ \cdot - a } \right\rangle = \left\langle ( \cdot - a ) h , \frac g{ \cdot - a } \right\rangle = \langle h , g \rangle . \]
It follows that $ h =(\textrm{const}) K_a $. Thus $ K_a \in \mathcal D $. Since $ a \in U $ is arbitrary this means that $ \mathcal D $ contains all $ K_a $'s, $ a \in U $, and is thus dense, a contradiction.
\end{proof}

A symmetric operator, $ L $, is called \textit{regular} if $ \hat\rho ( L ) = \mathbb C $.

\begin{lemma}\label{11}
$ A $ is regular, and has the deficiency indices $ n_\pm ( A ) = 1 $.
\end{lemma}

\begin{proof}
By the division property for any $ \lambda \in U $ the range of $ A -\lambda $ is $ \{ f \in H \colon f(\lambda ) = 0 \} $ which is a subspace in $ H $ of codimension $ 1 $ on account of the non-degeneracy assumption. Then, $ \ker ( A - \lambda ) = \{ 0 \} $ because $ \delta_\lambda \notin H $, see (i) above. Thus, $ U \subset \hat\rho (A) $ and $ n_\pm ( A ) = 1 $. 


It remains to show that $ \mathbb R \setminus U \subset  \hat \rho ( A ) $. Assume by contradiction that there exists a $ p \in \mathbb R \setminus U $ such that $ ( A - p ) f_n \longrightarrow 0 $ for some sequence $ f_n \in\mathcal D $ such that $ \| f_n \| = 1 $. Clearly, $ f_n $ converge to $ 0 $ in $ L^2 $ with respect to the measure $ \mu $ in the complement of any vicnity of $ p $. By the reproducing kernel property $ (t-p) f_n ( t ) \longrightarrow 0 $, and therefore $ f_n ( t ) \longrightarrow 0 $, for all $ t \in U $. On the other hand, $ \left|f_n ( t ) \right|^2 = \left| \langle f, K_t \rangle \right|^2 \le K(t,t) $ by the Schwarz inequality and therefore $ f_n $ converges to $ 0 $ locally in $ L^2 $ with respect to the measure $ \mu $ by the dominated convergence theorem. At this point we used assumption \eqref{tra}. Combined, these imply that $ f_n \longrightarrow 0 $ in $ H $, a contradiction. 
\end{proof}


A corollary of lemma \ref{11} is that $ A $ is simple and any selfadjoint extension of it has discrete spectrum, see e. g. \cite[Corollary 3.1, Chapter II]{GG} for the case when $ \mathcal D $ is dense in $ H $. A proof of this fact is given in lemma \ref{regular} at the end of this section. 

According to the following lemma nonzero elements of the deficiency subspaces of the operator $ A $ off the real line have no zeroes. 

\begin{lemma}\label{offR} $ \operatorname{Ran} ( A - w )^\perp \cap \operatorname{Ran} ( A - t ) = \{ 0 \} $ for any $ w \notin \mathbb R $, $ t \in \mathbb R $. In particular, for any nonzero $ \xi \in \operatorname{Ran} ( A - w )^\perp $ we have $ \xi ( x ) \ne 0 $ for all $ x\in U $.  
\end{lemma}

\begin{proof}
Let $ \xi \in \operatorname{Ran} ( A - w )^\perp \cap \operatorname{Ran} ( A - a ) $ for some $ a \in \mathbb R $. Then there exists an $ \eta \in \mathcal D $ such that $ \xi = ( A - a ) \eta $, and thus
\begin{equation}\label{xi}
\| \xi \|^2 = \langle ( A - a ) \eta , \xi \rangle =  ( w-a) \langle \eta, \xi \rangle .
\end{equation}  On taking the imaginary part we get $(\Im w ) \langle \eta , (A-a) \eta \rangle = 0 $. Since $ w $ is nonreal it means that $ \langle \eta , (A-a) \eta \rangle = 0 $. Plugging this back into \eqref{xi} we obtain $ \xi = 0 $, as required.  
\end{proof}

From now on we fix an arbitrary $ w \notin \R $ and $ \xi \in \operatorname{Ran} ( A - w )^\perp $. 

Consider the map $ f \mapsto f_\xi $ where $ f_\xi $ is the function of the complex argument $ \lambda $ defined by the requirement that
\begin{equation} f - f_\xi ( \lambda ) \xi \in \operatorname{Ran} ( A - \lambda )  , \label{deffxi}\end{equation}
that is, $ f_\xi ( \lambda ) $ is the component of $ f $ along $ \xi $ in the decomposition $ H = \operatorname{Ran} ( A - \lambda ) \dot{+} \mathcal L \{ \xi \} $.  
This formula defines the function $ f_\xi ( \lambda ) $ on the complement of the set
\[ S = \{ \lambda \in \C \colon \xi \in \operatorname{Ran} ( A - \lambda ) \} . \]

Notice that $ S \cap \R = \emptyset $ by lemma \ref{offR}. The following assertions are now facts of the general theory \cite[p. 10--11]{GG}, 

\medskip

(a) $ S $ is discrete in $ \C $;

(b) For any $ f \in H $ the function $ f_\xi $ is meromorphic in $ \C $, all poles of it are contained in $ S $, and for any $ z \in S $ there exists an $ N = N(z) $ such that the singularity of $ ( \cdot -z)^N f_\xi $ at $ z $ (if any) is removable for all $ f \in H $, that is, the multiplicities of poles of $ f_\xi $ at $ z $ are bounded uniformly in $ f \in H $.  

\medskip

For completeness we provide a proof of these facts. Let $ \tilde A $ be a selfadjoint extension of the operator $ A $. Define the function 
\[ \varphi ( z ) =  \bigl( \tilde{A} - \overline w \bigr) \bigl( \tilde{A} - z \bigr)^{ - 1} \xi . \]
This function is analytic in $ \rho ( \tilde{A} ) $ and does not vanish on this set. The point of considering it is the fact that 
\[ \varphi ( z ) \perp \operatorname{Ran} \left( A - \overline z \right) , \; z \in \rho ( \tilde{A} ) , \]
which is verified by the following computation. For any $ \eta \in \mathcal D $ we have
\begin{align*}
\left\langle ( A - \overline z ) \eta , \bigl( \tilde{A} - \overline w \bigr) \bigl( \tilde{A} - z \bigr)^{ - 1} \xi \right\rangle = \left\langle ( A - \overline z ) \eta , \xi + ( z - \overline w ) \bigl( \tilde{A} - z \bigr)^{ - 1} \xi \right\rangle = \\ \left\langle ( A - w ) \eta , \xi \right\rangle + ( w - \overline z ) \langle \eta , \xi \rangle + \left\langle ( A - \overline z ) \eta , ( z - \overline w ) \bigl( \tilde{A} - z \bigr)^{ - 1} \xi \right\rangle ,
\end{align*}
which is zero because the first term in the right hand side vanishes by the choice of $ \xi $ and the other two cancel each other because $ A $ in $ ( A - \overline z ) \eta $ in the third term can be replaced by the operator $ \tilde A  $. 

Taking scalar product of \eqref{deffxi} with $ \varphi ( \overline \lambda ) $ we obtain 
\begin{equation} \label{fxi} f_\xi ( \lambda ) \langle \xi , \varphi ( \overline \lambda ) \rangle = \left\langle f , \varphi ( \overline \lambda ) \right\rangle ,\; \lambda \notin S \cup \sigma ( \tilde A ) . \end{equation}
The operator $ \tilde A $ has discrete spectrum, hence the function $ \Psi ( \lambda ) = \langle \xi , \varphi ( \overline \lambda ) \rangle $ is meromorphic in $ \C $ with all poles lying in $ \sigma ( \tilde{A} ) $. It does not vanish identically because $ \Psi ( w ) = \| \xi \|^2 \ne 0 $, and the set of zeroes of $ \Psi $ in $ \rho ( \tilde A ) $ coincides with $ S $. It follows that $ S $ is discrete in $ \C $ and the function $ f_\xi $ is meromorphic with poles contained in $ \sigma ( \tilde A ) \cup S $. Notice now that the definition \eqref{deffxi} of $ f_\xi $ does not depend on the choice of the selfadjoint extension $ \tilde A $. Since $ A $ is simple and $ n_\pm ( A ) = 1 $, the spectra of different selfadjoint extensions $ \tilde A $ are disjoint (in fact, interlacing, but we do not need that). It follows that all poles of $ f_\xi $ are contained in $ S $. The function $ \Psi $ does not depend on $ f $, hence the multiplicity of a pole of $ f_\xi $ does not exceed the multiplicity of the zero of $ \Psi $ at the same point. Assertions (a) and (b) are established. Moreover, considering a nonzero $ f \in \operatorname{Ran} ( A - z )^\perp $ we obtain that for this $ f $ the right hand side in \eqref{fxi} is nonzero at $ \lambda = z $, and thus

\medskip

(c) The maximum of the multiplicity of the pole of $ f_\xi $ at a point $ z\in S $ over $ f \in H $ coincides with the 
 multiplicity of the zero of $ \Psi $ at $ z $. 
 
From now on we fix an arbitrary selfadjoint extension $ \tilde A $ of the operator $ A $. Let $ d\nu_\xi $ be the spectral measure of the operator $ \tilde A $ on $ \xi $. Then for any $ f , g \in H $
 \begin{equation}\label{sc-prod} \langle f , g \rangle = \int f_\xi ( x ) \overline{ g_\xi ( x )} d\nu_\xi ( x ) . \end{equation}

This assertion is a fact of general theory \cite[p. 49]{GG}. For completeness we provide an elementary proof of it in the situation under consideration, but postpone the proof until the end of this section.

Thus the map $ \mathcal W \colon f \mapsto f_\xi $ is an isomorphism of $ H $ onto a Hilbert space $ X $ of meromorphic functions which are analytic on the real line and square summable with respect to $ d \nu_\xi $-measure, endowed with the norm 
\[ \left\| u \right\|_X^2 : = \int \left|u(t)\right|^2 d\nu_\xi (t) . \] 

For $ \lambda \in U $ the definition \eqref{deffxi} gives
\[ f_\xi ( \lambda ) = \frac{ f( \lambda )}{ \xi ( \lambda ) } . \]
We have thus constructed a function $ \xi $ and a discrete set $ S $ such that the function $ f/\xi $ admits meromorphic continuation for all $ f \in H $, poles of this continuation are in $ S $, and the corresponding map is a Hilbert space isomorphism. The next step would be to show that the multiplication by an entire function with $ S $ being the set of its zeroes turns $ X $ into a de Branges space. 

In this way, we are first going to show that the space $ X $ itself satisfies the division property off the set $ S $. 

\begin{lemma}\label{divX}
Let $ z \notin S $. Then for any $ u \in X $ such that $ u ( z ) = 0 $ the function $ u/(\cdot - z ) \in X $.
\end{lemma}

\begin{proof}
Let $ u \in X $, $ u(z) = 0 $. Then $ u = f_\xi $ for some $ f \in H $, that is, for any $ \lambda \notin S $
\[ f = ( A-\lambda) g + u(\lambda ) \xi \] 
for some $ g \in \mathcal D $. In particular, $ f = ( A-z) h $ for some $ h \in \mathcal D $. Equalizing we find for $ \lambda \ne z $, $ \lambda \notin S $,
\[ h = ( A -\lambda ) \left[ \frac{ g-h}{\lambda -z } \right] + \frac{u(\lambda)}{\lambda - z } \xi , \]
that is, $ u/(\cdot - z ) = h_\xi \in X $.
\end{proof}

Let $ R $ be the Weierstra\ss\, canonical product for the set $ S$ where the multiplicity of each zero $ z \in S $ is taken to be the maximum of multiplicities of poles at $ z $ of functions from $ \operatorname{Ran} \mathcal W $ (by (c) the maximum exists and coincides with the multiplicity of the corresponding zero of the function $ \Psi $).  

Let $ \mathcal X = R X $. Then $ \mathcal X $ is the Hilbert space with respect to the scalar product 
\begin{equation}\label{norm} \left\langle x , x^\prime \right\rangle_{ \mathcal X } = \int_\R x ( t ) \overline{ x^\prime ( t ) } \frac{ d\nu_\xi ( t)}{\left| R ( t ) \right|^2 } .\end{equation} 
Recall that $ R(t) \ne 0 $ for real $ t $ as $ S \cap \R = 
\emptyset $.  
By construction, the elements of $ \mathcal X $ are entire functions. 

\begin{lemma}
$ \mathcal X $ is a reproducing kernel Hilbert space.
\end{lemma} 

\begin{proof}
For $ z\notin S $ the value of $ R f_\xi $ at $ z $ reads off from \eqref{fxi} to be\footnote{if $ z \in \R $ one should take the selfadjoint extension $ \tilde A $ for which $ z \notin \sigma ( \tilde A ) $ in the definition of the function $ \varphi $.}
\[ ( R f_\xi ) ( z ) = \left\langle R f_\xi , \frac{ \overline{ R(z)}}{\overline{ \Psi ( z ) }} R \mathcal W \varphi ( \overline z ) \right\rangle_{\mathcal X } . \]
For $ z \in S $ we take into account that the function $ R / \Psi $ is analytic at $ z $ by the definition of $ R $ and similarly have 
\[  ( R f_\xi ) ( z ) = \left(\frac R\Psi \right) (z ) \left\langle R f_\xi , R \mathcal W \varphi ( \overline z ) \right\rangle_{\mathcal X } . \] 
\end{proof}

The division property of lemma \ref{divX} is obviously inherited by the space $ \mathcal X $. Two things still miss for $ \mathcal X $ being a de Branges space -- as yet, the division property is not established for $ z \in S $, and there is no invariance with respect to the conjugation. To get through
we are going to use a particular case of \cite[Theorem 2.1]{BLMS} formulated as follows. 

\begin{theorem}
 \label{integrable}
Let $ Y $ be a reproducing kernel Hilbert space of functions on the real line which is isometrically embedded in $ L^2 ( \R , \mu ) $ for some measure $ \mu $. Assume  that for any $ k \in \R $ and $ f \in 
Y $ satisfying  $ f ( k ) = 0 $ there exists a unique function $ g \in Y $ such that   
$ f ( x ) =  ( x - k ) g ( x )$
for all $ x \in \R $. Then there exist functions $A$, $B$ defined on $\R $ such that for all $x,y\in \R $, $ x \ne y $, the reproducing kernel $ K $ of the space $ Y $ admits the integrable representation 
\begin{equation} 
K(x,y)=\displaystyle \frac{A(x)\overline{B(y)} - B(x)\overline{A(y)}}{x-y}.
\end{equation}
\end{theorem}

Let us apply this theorem to the space $ Y $ made of restrictions of the functions from $ \mathcal X $ to the real line endowed with the norm \eqref{norm}. According to it, there exist functions $ A , B $ on the real line such that the reproducing kernel $ K^{\mathcal X }_y $ of the space $ \mathcal X $ at a point $ y \in \R $ has the form  
\begin{equation} \label{integ}
K_y^{ \mathcal X } (x) =\displaystyle \frac{A(x)\overline{B(y)} - B(x)\overline{A(y)}}{x-y} , \; x \in \R, \; x\ne y .
\end{equation}
Here we took into account that for $ y \in \R $ the reproducing kernel of the space $ Y $ at the point $ y $ is obviously the restriction of $ K^{\mathcal X }_y $ to the real line. 

The function $ K^{\mathcal X }_y $ is entire for any $ y \in \R $. This implies that functions $ A $ and $ B $ are entire as well. Indeed, one can solve \eqref{integ} for $ A $ and $ B $ using two different values of $ y $ expressing them as linear combinations of entire functions $ K_y^{ \mathcal X } (\cdot ) (\cdot - y ) $, unless $ A \equiv (\textrm{const}) B $, and thus $ K^{\mathcal X }_y = (\textrm{const}) B(x) \overline{ B(y)} \left( x-y \right)^{ -1 } $ which means that $ B(y ) = 0 $ for all $ y $, an absurd conclusion.

Now, for all $ y \in \C $ we have
\begin{equation} \label{integC}
K_y^{ \mathcal X } (x) =\displaystyle \frac{A(x)\overline{B(y)} - B(x)\overline{A(y)}}{x-\overline y} , \; x \in \R, \; x\ne y .
\end{equation}
This follows from the fact that for a fixed $ x \in \R $ the kernel $  K_y^{ \mathcal X } (x)  $ is an antiholomorphic function of $ y $ (just because $ K_y^{ \mathcal X } (x)  = \overline{ K_x^{ \mathcal X } (y) } $ for all $ x , y \in \C $), the right hand side in \eqref{integC} is antiholomorphic in $ y $ as well, and the two coincide for $ y \in \R $ by \eqref{integ}. Using the analyticity once more, this time in $ x $ for fixed $ y \in \C $, we conclude that \eqref{integC} holds for all $ x\in \C $, $ x\ne \overline y $, and that the numerator there vanishes at $ x = \overline y $, 
$ A ( \overline y ) \overline{ B( y ) } = \overline{ A(y) } B ( \overline y ) $, or, using the notation $ F^* (z ) := \overline{ F ( \overline z ) } $, 
\[ A^* ( y ) B ( y ) = A ( y ) B^* ( y ) , \; y \in \C . \] 
This equality means that nonreal zeroes of $ A $ form conjugate pairs, and the multiplicities of zeroes at conjugate points are equal. Indeed, otherwise the functions $ B $ and $ A $ would have a common zero, say $ a $, the reproducing kernel $ K^{\mathcal X}_a = 0 $, so all functions from $ \mathcal X $ would vanish at $ a $. For $ a \notin S $ this contradicts the division property of lemma \ref{divX}, and for $ a \in S $ it suffices to take a nonzero $ f \in \operatorname{Ran}\left( A-a \right)^\perp $ to get a contradiction, as $ (R f_\xi ) ( a ) \ne 0 $ for such an $ f$ by construction. The same assertion holds for the function $ B$. 

It follows that $ A^* / A = B^* / B $ extends to a zero-free entire function, which thus has an entire square root. Let 
\[ \Omega = \sqrt{ \frac{ A^*}A } , \]
then $ \Omega $ is an entire function satisfying $ \Omega \Omega^* = 1 $. 

Define now the Hilbert space $ \mathcal H = \Omega \mathcal X $ with the norm $ \left\| \Omega h \right\|_{ \mathcal H } := \left\| h \right\|_{ \mathcal X } $, $ h \in \mathcal X $. It is a reproducing kernel Hilbert space of entire functions isometrically embedded in $ L^2 ( \R , d\mu ) $ for some measure $ \mu $, and satisfying the property that for any $ u \in \mathcal H $ and $ z \notin S $ such that $ u ( z ) = 0 $ the function $ u/(\cdot - z ) \in \mathcal H $. All these properties of $ \mathcal H $ are inherited from $ \mathcal X $. The reproducing kernel $ K^{ \mathcal H }_y $ of $ \mathcal H $ admits representation \eqref{integC} with \textit{real} entire functions $ A $ and $ B $. Explicitly, for any $ y \in \C $, $ t \in \R $
\begin{equation}\label{Kh} K_y^{ \mathcal H } ( t ) = \overline{ \Omega ( y )} K_y^{ \mathcal X } ( t ) \Omega ( t ) = \frac{A^\kappa (t)\overline{B^\kappa (y)} - B^\kappa (t)\overline{A^\kappa (y)}}{t-\overline y} ,\end{equation}
\[ A^\kappa := \Omega A , \; B^\kappa := \Omega B . \]
By the definition of $ \Omega $ the arguments of $ A^\kappa (t) $ and $ B^\kappa (t) $ are $ 0 \mod \pi $, and thus $ A^\kappa (t) $ and $ B^\kappa (t) $ are real for real $ t $. 

Let $ E = A^\kappa + i B^\kappa $. Then $ E $ is an Hermite--Biehler function. This immediately follows from plugging $ t = y $ in \eqref{Kh} and the fact that $ K_y^{ \mathcal H } ( y ) = \left\| K_y^{ \mathcal H } \right\|^2 > 0 $ for all complex $ y $, the property inherited by $ \mathcal H $ from the space $ \mathcal X $. 

Let $ \mathcal H ( E ) $ be the de Branges space corresponding to the function $ E $. Then $ \mathcal H ( E ) $ coincides with $ \mathcal H $ because they have the same reproducing kernels by construction. This proves theorem \ref{main} with the function $ \Phi $ of the form $ \Phi = \frac \xi{ R \Omega } $.

It remains to prove the postponed assertions -- discreteness of the spectrum of selfadjoint extensions $ \tilde A $, and formula \eqref{sc-prod}.

\begin{lemma}\label{regular} Let $ T $ be a regular closed symmetric operator in a Hilbert space $ \mathfrak H $ with the domain $ \mathcal D $ such that $ n_\pm ( T ) = 1 $. Then any selfadjoint extension of $ T $ has discrete spectrum.
\end{lemma} 

\begin{proof} Consider a selfadjoint extension, $ \tilde T = \tilde T^* \supset T $. Any selfadjoint extension of $ T $ is proper, hence for any $ \lambda \in \R $ the linear set $ \operatorname{Ran} ( \tilde T - \lambda ) $ is either $ \mathfrak H $, or coincides with $ \operatorname{Ran} ( T - \lambda ) $. This means that any point $ \lambda \in \sigma ( \tilde T ) $ is a simple eigenvalue of $ \tilde T $. Assume by contradiction that a sequence of eigenvalues $ \lambda_n $ has a finite limit, $ \lambda_n \to \lambda $. Let $ e_n $ be the corresponding eigenfunctions. The domain of $ \tilde T $ is a linear sum of $ \mathcal D $ and a one-dimensional subspace spanned by a nonzero vector, $ \xi \in 
\mathfrak H $. Fix the normalization of $ e_n $ assuming that $ e_n = f_n +\xi $, $ f_n \in \mathcal D $. Then 
\[ ( \tilde T - \lambda_n ) e_n = ( T - \lambda_n ) f_n + ( \tilde T - \lambda_n ) \xi = 0 . \]
Thus, $ ( T - \lambda_n ) f_n $ has a finite limit, $ ( \lambda - \tilde T ) \xi $. 
Now, $ \lambda \in \hat \rho ( T ) $ means that there exists an $ \epsilon > 0 $ such that $ \| ( T - w ) u \| \ge \epsilon \| u \| $, $ u \in \mathcal D $, for all $ w $ sufficiently close to $ \lambda $, therefore the norms of $ f_n $ are bounded, hence $ ( T - \lambda ) f_n $ also converges, and so does $ f_n $. Let $ f = \lim f_n $, then $ e_n \longrightarrow f + \xi $. Since $ e_n $'s are mutually orthogonal, this implies that $ \xi + f = 0 $. On the other hand, $ f \in \mathcal D $ by closedness of $ T $. We infer that $ \xi \in \mathcal D $ which is the required contradiction. \end{proof}

To establish \eqref{sc-prod} it suffices to verify that 
\\
(A) $ \int \left| f_\xi \right|^2 d\nu_\xi = \| f \|^2 $ for any eigenfunction $ f $ of $ \tilde A $, \\
(B) $ \int f_\xi \overline{ g_\xi } d\nu_\xi = 0 $ when $ f $ and $ g $ are eigenfunctions of $ \tilde A $ corresponding to different eigenvalues. 

For any eigenvalue $ \lambda $ of $ \tilde A $ we have $ \operatorname{Ran} ( \tilde A - \lambda ) = \operatorname{Ran} ( A-\lambda ) = \ker \bigl( \tilde A - \lambda \bigr)^\perp $. From definition \eqref{deffxi} of $ f_\xi $ we conclude that if $ f \in \ker (\tilde A - z ) $ then $ f_\xi ( \lambda ) = 0 $ for all $ \lambda \in \sigma ( \tilde A ) $, $ \lambda \ne z $. This immediately implies (B). For $ \lambda = z $, on taking scalar product of \eqref{deffxi} with $ f $ we find that $ f_\xi ( z ) = \left\| f \right\|^2 / \langle \xi , f \rangle $. Since $ \nu_\xi \{ z \} = | \langle \xi , f \rangle |^2 / \left\| f \right\|^2 $ this implies (A).  

%

\section{Proof of Theorem \ref{main-inv}}

Under the assumptions of theorem \ref{main-inv} one can apply theorem \ref{main}. The reproducing kernels of the spaces $ H $ and $ \mathcal H ( E ) $ are related as follows,
\[ K_y^H (t) = \overline{\Phi ( y )} \Phi ( t ) K_y^{ \mathcal H ( E ) } (t) , \; y \in U . \] 
The assumption of reality of $ H$ means that the argument of $ \Phi $ is constant hence one can assume $ \Phi $ to be real (for the multiplication by a unimodular constant preserves the assertion of theorem \ref{main}), thus 
\begin{equation}\label{KH-deBr} K_y^H (t) = \Phi ( y ) \Phi ( t ) K_y^{ \mathcal H ( E ) }(t) , \; y \in U . \end{equation} 
This proves the first assertion of theorem \ref{main-inv}. For the uniqueness part assume that \eqref{KH-deBr} is satisfied for two distinct pairs of the function $ \Phi $ and the de Branges space $ \mathcal H (E) $, $ ( \Phi_1 , \mathcal H ( E_1) ) $, and $ ( \Phi_2 , \mathcal H ( E_2) ) $, 
\[ \Phi_1 ( y ) \Phi_1 ( t ) K_y^{ \mathcal H ( E_1 ) }(t) =  \Phi_2 ( y ) \Phi_2 ( t ) K_y^{ \mathcal H ( E_2 ) }(t) . \]
The analyticity of the reproducing kernel of a de Branges space implies that $ \Phi_1 / \Phi_2 $ extends from $ U $ to a real meromorphic function all poles of which are contained in the set of zeroes of the function $ K_y^{ \mathcal H ( E_1 ) } $. Since $  \Phi_1 / \Phi_2 $ does not depend on $ y $, the poles of it are only possible at points $ t \in \R $ such that $ K_y^{ \mathcal H ( E_1 ) } (t) = 0 $ for all $ y \in U $. By symmetry this means that $ K_t^{ \mathcal H ( E_1 ) } (y) = 0 $ for all $ y \in U $, hence $ K_t^{ \mathcal H ( E_1 ) } = 0 $, which is impossible as $ E_1 $ has no real zeroes. It follows that the function $ \Phi_1 / \Phi_2 $ is actually entire. It is zero-free since $ \Phi_1 $ and $ \Phi_2 $ can be interchanged in the argument. Theorem \ref{main-inv} is proven. 

\subsection{Extensions and modifications}

A. There is a variant of theorem \ref{main} in which the set $ U $ is supposed to be closed, while all the requirements on the space $ H $ are intrinsic (no emebedding into an $ L^2 $ space is assumed). 

\begin{theorem}\label{main-var}
Let $ H $ be a non-degenerate reproducing kernel Hilbert space of functions on a closed set $ U \subset \mathbb R $ satisfying the division property. Assume that the space $ H $ obeys the following conditions,

\begin{itemize}
\item (symmetry of multiplication) the equality
\[ \left\langle tf, g \right\rangle_H = \left\langle f, tg \right\rangle_H \] 
is satisfied for any $ f,g \in H $ such that $ tf, tg \in H $;
\item (normality) $ \left\| f \right\|_H \le \left\| g \right\|_H $ for any $ f,g \in H $ such that $ | f(x) | \le | g(x) | $ for all $ x \in U $. 
\end{itemize}

Then the conclusion of theorem \ref{main} holds. 
\end{theorem}

\begin{proof}
The only point in the proof of theorem \ref{main} where the finite trace assumption \eqref{tra} was used is the inclusion $ \mathbb R \setminus U \subset \hat\rho ( A ) $ in the proof of lemma \ref{11}. The proof that $ U \subset \hat{\rho } ( L ) $ depends on the division property only, hence this inclusion holds in the situation under consideration. Since $ U $ is closed, for any $ \lambda \in \mathbb R \setminus U $ there exists an $ \varepsilon > 0 $ such that for any $ f \in \mathcal D $ 
\[ | ( x - \lambda ) f ( x ) | \ge \varepsilon | f(x)| , \; x \in U . \]
By the normality assumption then
\[ \| ( A - \lambda ) f \| \ge \varepsilon \| f \| , \]
that is, $ \lambda \in \hat\rho ( A ) $.
\end{proof}

Let us comment on the normality assumption. In the theory of determinantal processes we are interested in spaces $ H $ isometrically embedded into $ L^2 ( U , d\mu ) $ for a Borel measure $ \mu $. The normality assumption in theorem \ref{main-var} is obviously necessary for the existence of such an embedding. The following example shows that it is independent of the division property and the symmetricity of the multiplication.
  
Let $ H $ be the Hilbert space elements of which are restrictions of the functions from the Paley--Wiener space $ PW_\pi $ to the segment $ [-1,1] $. For the norm of an element in $ H $ we take the $ L^2 ( \R ) $-norm of its analytic extension.  

To see that the normality fails for $ H $ consider 
\[ e_n = \left. \frac{ \sin \pi x }{ x-n } \right|_{ [-1,1] } .\]
Clearly, $ | e_n ( x ) | \le \frac 1{n-1} | e_0 ( x ) | $, for $ n \ge 1 $, so $f =  ( n-1 ) e_n $, $ g = e_0 $ satisfy $ | f(x)|\le |g(x) | $ but $ \| f \| = (n-1) \| g \| $. All the other assumption of theorem \ref{main-var} are satisfied for $ H $, as they are inherited from the Paley--Wiener space. It is interesting to notice that the assertion of the theorem holds for the space $ H $ despite the lack of normality.


B. Theorem \ref{main} can be slightly strengthened by noticing that under  assumption \eqref{tra} the division property on a dense set implies the full division property. This observation is not used here but may be of independent interest.    

In the following proposition $ H $ is a non-degenerate reproducing kernel Hilbert space of functions on a Borel set $ U \subset \mathbb R $ such that  the space $ H $ does not contain a nonzero function supported at a single point. Let $ K(x, y ) $ be the reproducing kernel of $ H $.

\begin{proposition}\label{main-dense}
 Assume that 
\begin{itemize}
\item $ H $ is isometrically embedded into the space $ L^2 ( U, d\mu ) $ for some Borel measure $ \mu $ on $ U $, and 
\[ \int_{I \cap U} K ( t ,t ) d\mu ( t ) < \infty \] for any bounded interval $ I $.

\item for any $ p \in U $ the set $ \mathcal M_p = \{ h \in H : \exists f \in H \colon h = ( \cdot - p ) f \} $ is dense in $ H_p = \{ f\in H \colon f(p) = 0 \} $. 
\end{itemize} 
Then $ \mathcal M_p = H_p $ and thus the assumptions of theorem  \ref{main} are satisfied.
\end{proposition}

\begin{proof}
Let $ h \in H_p $. We have to show that there exists an $ f \in H $ such that $ h = ( \cdot - p ) f $. By the density assumption, there exists a sequence $ f_n \in \mathcal D $ such that $ ( A - p ) f_n \longrightarrow h $. Obviously, the sequence $ f_n $ converges in $ L^2 $ with respect to the measure $ \mu $ in a complement to any open vicinity of $ p $.

Let us show that the sequence $ f_n $ is bounded in $ H$. Assume by contradiction that $ \| f_{ n_k } \| \longrightarrow \infty $ for some subsequence $ n_k $. Let $ \xi_k = f_{n_k}  / \| f_{n_k } \| $. Then $ ( A - p ) \xi_k \longrightarrow 0 $. Therefore $ \xi_k ( t ) \longrightarrow 0 $ for any $ t \in U $, $ t\ne p $. On the other hand $ \| \xi_k \| = 1 $, and so $ \left| \xi_k ( t ) \right|^2 \le K( t,t ) $ for any $ t \in U $. By the dominated convergence and the isometric embedding into $ L^2 ( U , d\mu ) $, we obtain that $ \xi_k \longrightarrow 0 $ in $ H $, provided that $ \mu \{ p \} = 0  $, a contradiction. If $ \mu \{ p \} \ne 0 $ then $ \xi_k - \xi_k ( p ) \delta_p \longrightarrow 0 $, hence either $ \delta_p \in H $ which contradicts the assumptions, or $ \xi_k \longrightarrow 0 $, again a contradiction. Thus, $ \sup_n \| f_n \| < \infty $. 

Fix a weakly convergent subsequence, $ f_{ n_k } $. Let $ f $ be its limit.
By the reproducing kernel property this subsequence converges pointwise on $ U $. In particular, $ ( t-p ) f_{ n_k } ( t) \longrightarrow (t-p)f(t) $ for all $ t \in U $, and therefore $ h ( t ) = ( t-p) f( t ) $.
\end{proof}

\section{Examples}

$1^\circ$. Let $ H $ be the de Branges space corresponding to an Hermite--Biehler function $ E $ without real zeroes. The assertion of theorem \ref{main} is trivial then, but let us follow the construction. In this case $ \xi = K_w $, $ K_w $ being the reproducing kernel of the space $ H $ at the point $ w $, $ S $ is the set of zeroes of the function $ K_w $, $ f_\xi = f/K_w $, $ R $ is the canonical product over the zeroes of $ K_w $, and thus $ R f_\xi = T f $ where $ T $ is an entire function without zeroes. The reproducing kernel of the space $ \mathcal X $ at a point $ y \in \C $ is $ K_y^{ \mathcal X } = \overline{ T(y) } K_y T $, hence $ \Omega = \sqrt{ T^* / T } $ and thus the isomorphism $ \Phi $ is given by the multiplication on the zero-free entire function $ \left( T^* T \right)^{ -1/2 } $. In the case of the Paley--Wiener space this entire function is constant, for $ K_w $ coincides with its canonical product.

This example includes the sine- and Airy- processes \cite{Dyson,TW-Airy}.

$ 2^\circ $. (discrete sine process \cite{BOO,johansson}). Fix an arbitrary $ b \in (0, \pi/2 ) $, and let $ U= \mathbb Z $, $ K ( m , n ) = \frac{ \sin( b ( m-n ))}{\pi (m-n)} $, $ m , n \in \mathbb Z $. This kernel defines a projection in $ l^2 ( \mathbb Z ) $ on a subspace $ H $ corresponding to a quasi-invariant determinantal process \cite{BOO}. Let $ W : l^2 ( \mathbb Z ) \to L^2 ( -\pi , \pi ) $ be defined by $ W a = \sum a_n e^{ inx } $, $ a = \{ a_n \} $. It is clear that $ W H $ coincides with the subspace of functions supported on $ [-b,b ] $, that $ W \mathcal D = \{ u \in H^1 ( -\pi , \pi ) : u ( x ) = 0 \textrm{ for a.e. } x , |x| > b \} $ where $ H^1 $ stands for the Sobolev class, and that $ W A W^{-1 } $ is the operator $ -i d/dx $. For any $ z \in \C $ the subspace $ W \operatorname{Ran} ( A - z )^\perp $ is spanned by the function $ e^{ i \overline z x } \chi_{[-b, b ] } $. 
In particular the element $ \xi $ can be chosen to be the sequence of the Fourier coefficients of the function $ e^{ i \overline w x } \chi_{[-b, b ] } $. Then $ W\varphi ( z ) = c( z) e^{ izx} $, where $ c $ is an analytic function off a discrete set on the real axis, whose exact form is not required since it cancels out in calculations of $ f_\xi $ via formula \eqref{fxi}. We have,
\begin{align*} f_\xi ( \lambda ) =  \frac{\left\langle f , \varphi ( \overline \lambda ) \right\rangle }{\langle \xi , \varphi ( \overline \lambda ) \rangle } = \frac{\left\langle Wf , W\varphi ( \overline \lambda ) \right\rangle }{\langle W\xi , W\varphi ( \overline \lambda ) \rangle } = \frac{\left\langle Wf , e^{ i \overline \lambda x } \right\rangle }{\left\langle e^{ i \overline w x } , e^{ i \overline \lambda x } \right\rangle_{ L^2 ( -b,b)} } = \\ \frac{ \overline w - \lambda }{ 2\sin ( b (\overline w - \lambda )) } \sqrt{ 2\pi } ( \widehat{Wf}) (\lambda ) . \end{align*}
Here the function $ Wf $ is supposed to be extended by zero to the whole real line, and 
$ \; \widehat{} \; \; $ stands for the standard Fourier transform in $ L^2 ( \R ) $.
 
The set $ S $ is thus the set of $ \lambda \ne \overline w $ such that $ \sin (b( \overline w - \lambda )) = 0$. Since the function $ \frac{\sin z}z $ coincides with its canonical product, in the situation under consideration  
\[ R f_\xi = \frac 1b \sqrt{ \frac\pi{2} } ( \widehat{Wf}) (\lambda ). \]
The resulting space $ \mathcal X $ coincides with the Paley--Wiener space $PW_b $ as a set, and the norm in $ \mathcal X $ is a multiple of the norm in $ PW_b $. Since $ \mathcal X $ is already a de Branges space there is no need to use the $ \Omega $-symmetrization here. 

Notice that the same de Branges space -- the Paley-Wiener space -- arises in the two previous examples from two different determinantal processes (sine and discrete sine). This does not contradict uniqueness, for the respective sets $ U $ are different.

$ 3^\circ $. (Bessel process \cite{TW-Bessel}). Let $ U = \R_+ $, 
\[ K ( x , y ) = \frac{\sqrt x J_{s+1} ( \sqrt x ) J_s ( \sqrt y ) - \sqrt y J_{s+1} ( \sqrt y ) J_s ( \sqrt x )}{2(x-y)} ,\]
with $ H = \operatorname{Ran} K $ in $ L^2 ( \R_+ ) $. Here $ s > -1 $ is a parameter, and $ J_s $ stands for the Bessel function. In this case $ H $ is a reproducing kernel Hilbert space of analytic functions on $ \C\setminus \{ t, t\le 0 \} $ such that $ x^{ -s/2 } f $ is an entire function for all $ f \in H $. In particular, $ f_\xi = f/ K_w $, where $ K_w $ is the reproducing kernel of $ H $ at the point $ w $. The function $ K_w (x) x^{ -s/2 } $ is an entire function of order $ 1/2 $, because such is the function $ J_s ( \sqrt x ) x^{ -s/2} $, and therefore it coincides with its canonical product, hence $ R f_\xi =  x^{ -s/2}f R / ( x^{ -s/2 }K_w ) = f x^{ -s/2 } $. The reproducing kernel of the space $ \mathcal X $ at $ y\in \R $ has the form
\[ K_y^{ \mathcal X }(t ) = y^{ -s/2 } K_y (t) t^{ -s/2 } .\] The kernel is real, the $ \Omega $-symmetrization is not required, so $ \mathcal X $ is the de Branges space $ \mathcal H ( E ) $ with 
\[ E ( t ) = \frac\pi{\sqrt 2} \bigl( t j_{s+1} (\sqrt t ) + i j_s (\sqrt t ) \bigr) ,\;  j_s (t ) := J_s(t) t^{ -s} . \]

\section{Acknowledgements}
The author learnt about the problem studied in this paper from  Alexander Bufetov. This and many helpful discussions are gratefully acknowledged. The author thanks the unknown referees for their suggestions improving the presentation. The work was supported by the Russian Science Foundation under grant no. 22–11–00071.

\section{Conflict of interest}

The author declares that he has no conflict of interest.

\end{document}